\documentclass[11pt,mathserif]{amsart}
\usepackage[utf8]{inputenc}
\usepackage{eulervm,ulem}
\usepackage{amsmath}
\usepackage{amsthm}
\usepackage{import}
\usepackage{geometry}
\usepackage{mathrsfs} 
\usepackage{amsfonts}
\usepackage{amssymb}
\usepackage{palatino,color}
\usepackage{hyperref}
\usepackage{float}
\usepackage[font=small,skip=-10pt]{caption}
\usepackage{comment,accents}
\usepackage[export]{adjustbox}
\usepackage{enumerate,graphicx,enumitem,enumerate}

\newtheorem{prop}{Proposition}[section]
\newtheorem{theorem}{Theorem}[section]
\newtheorem{lemma}{Lemma}[section]
\newtheorem{corollary}{Corollary}[section]

\newtheorem{remark}{Remark}[section]

\newcommand{\trace}{{\rm trace\,}}

\usepackage{amsmath,amssymb,euscript,yfonts,psfrag,latexsym,dsfont,graphicx,bbm,color,amstext,wasysym,subfig,flushend,parskip,soul,bm}
\usepackage{amsthm}
\usepackage[dvipsnames]{xcolor}

\graphicspath{{./},{./figures/}}

\newcommand{\ignore}[1]{}

\newcommand{\black}{\color{black}}

\definecolor{grey}{rgb}{0.6,0.3,0.3}
\definecolor{lgrey}{rgb}{0.9,.7,0.7}

\def\spacingset#1{\def\baselinestretch{#1}\small\normalsize}
\spacingset{1}
\renewcommand{\em}{\it}

\title[On a Fej\'er-Riesz factorization]{On a Fej\'er-Riesz factorization\\
of generalized trigonometric polynomials}
\author{Tryphon T. Georgiou and Anders Lindquist}
\thanks{T.T.\ Georgiou is with the Department of Mechanical \& Aerospace Engineering,
University of California, Irvine, email: {tryphon@uci.edu}}
\thanks{A.\ Lindquist is with the Department of  Automation and the School of Mathematical Sciences, Shanghai Jiao Tong University, Shanghai, China, and with the Dept.\ of Mathematics,
KTH Royal Institute of Technology, Stockholm, Sweden, email: {alq@kth.se}}

\begin{document}
	\maketitle
	
	\begin{center}
{\em Dedicated to Tyrone Duncan on his 80th birthday}
	\end{center}
\begin{abstract}
Function theory on the unit disc proved key to a range of problems in statistics, probability theory, signal processing literature, and applications, and in this, a special place is occupied by trigonometric functions and the {\em Fej\'er-Riesz theorem} that non-negative trigonometric polynomials can be expressed as the modulus of a polynomial of the same degree evaluated on the unit circle. In the present note we consider a natural generalization of non-negative trigonometric polynomials that are matrix-valued with specified non-trivial poles (i.e., other than at the origin or at infinity). We are interested in the corresponding spectral factors and, specifically, we show that the factorization of trigonometric polynomials can be carried out in complete analogy with the Fej\'er-Riesz theorem.
The affinity of the factorization with the Fej\'er-Riesz theorem and the contrast to classical spectral factorization lies in the fact that the spectral factors have degree smaller than what standard construction in factorization theory would suggest. We provide two juxtaposed proofs of this fundamental theorem,
albeit for the case of strict positivity,
\black
one that relies on analytic interpolation theory and another that utilizes classical factorization theory based on the Yacubovich-Popov-Kalman (YPK) positive-real lemma.
\black
\end{abstract}

\renewcommand{\sigma}{V}

\section{Introduction}
The classical 
Fej\'er-Riesz theorem states \cite[Section 1.12]{GZ} that any nonnegative trigonometric polynomial in $z$,
\begin{align*}
p(e^{i\theta})&= \Lambda_0 +2\sum_{k=1}^{n-1} a_k\cos(k\theta)+b_k\sin(k\theta)\\
&=:\sum_{k=-(n-1)}^{n-1} c_k e^{ik\theta}
\end{align*}
for $\Lambda_0\in\mathbb R$ and $\Lambda_0>0$, and $c_k=a_k+ib_k\in\mathbb C$, $k\in\{1,2,\ldots,k-1\}$, i.e., one such that
\[p(e^{i\theta})\geq 0, \mbox{ for all }\theta\in[-\pi,\pi],
\]
can be written as the square of the modulus of a polynomial 
\[
g(z)=g_0+g_1z+\ldots +g_{n-1}z^{n-1}
\]
in $z$ of {\em equal degree}, where $z$ is on the unit circle, $z=e^{i\theta}$. 
The roots of $g(z)$ in the Fej\'er-Riesz factorization can be selected outside the unit disc \cite[Section 1.12]{GZ}. 
\black

Our aim in this work is derive a natural generalization of the  Fej\'er-Riesz factorization that applies to matrix-valued para-conjugate Hermitian rational functions as stated below. The statement is provided for positive generalized trigonometric polynomials while the more general case will be dealt with in future work.

\begin{theorem}{\bf Generalized Fej\'er-Riesz theorem:} \label{thm:factorization}
Let $\Lambda$ be a Hermitian matrix in $\mathbb C^{n\times n}$, $(A,B)\in \mathbb C^{n\times n}\times \mathbb C^{n\times m}$ a reachable pair of matrices with $B$ having rank $m$,
and the spectrum of $A$ in the open unit disc, \black and let
\begin{equation}\label{eq:kernel}
G(z):=(I-zA)^{-1}B.
\end{equation}
If
\begin{equation}\label{eq:correctform}
p(z)=G(z)^*\Lambda G(z)
\end{equation}
is positive definite for all $z=e^{i\theta}$ with $\theta\in[0,2\pi]$, then there exists a factorization
\begin{equation}\label{eq:factor}
p(z)=\underbrace{G(z)^*C^*}_{\sigma(z)^*}\Omega \underbrace{CG(z)}_{\sigma(z)},
\end{equation}
for a unique pair $(C,\Omega)$, with $C\in \mathbb R^{m\times n}$  and $\Omega\in\mathbb S^{m\times m}_+$ satisfying
\begin{itemize}
\item[i)] $CB=I$,
\item[ii)] $\Omega$ is a positive definite matrix,
\item[iii)] the polynomial $\det(\sigma(z))$, with $\sigma(z):=CG(z)$, has no root in the closed unit disc.
\black 
\end{itemize}
%
\end{theorem}

Thus, in the case of the Fej\'er-Riesz theorem, $p(z)$ is of the form \eqref{eq:correctform} with
\[
\Lambda =\left[\begin{matrix} 
\Lambda_0 & \Lambda_1 &\ldots &\ldots& \Lambda_{n-1}\\
\Lambda_{-1} & \Lambda_0 & \Lambda_1&& \Lambda_{n-2}\\
 \vdots &\Lambda_{-1}&\Lambda_0&\ddots&\vdots\\\vdots&&\ddots&\ddots&\Lambda_1 \\
\Lambda_{n-1}& \ldots &\ldots&\Lambda_{-1}& \Lambda_0 \end{matrix}\right],
\]
having a Toeplitz structure and $\Lambda_k=\frac{c_k}{n-k}\in\mathbb C$ for $k\in\{1,\ldots,n-1\}$, while $A=[A_{ij}]$ is an $n\times n$ companion matrix ($A_{ij}=1$ when $j-i=1$ and zero otherwise), $B=[1,\, 0,\ldots,\,0]^\prime $, and
$G(z)=\left[\begin{matrix}1, & z,& \ldots ,&z^{n-1}\end{matrix}\right]^\prime$.
Throughout $^\prime$ denotes {\em ``transposition''}, and $^*$ denotes {\em ``complex conjugation and transposition''} as well {\em as ``para-conjugation''}\footnote{The notation $G^*(z^{-1})$ indicates transposition and conjugation taking place on the coefficients of $G(z)$, while $z$ is replaced by $z^{-1}$. Thereby, $G(z)^*\mid_{z=e^{i\theta}}= G(e^{i\theta})^*$.}, i.e., $G(z)^*:=G^*(z^{-1})$ and, later on, the adjoint of linear operators.
The statement of the Fej\'er-Riesz theorem asserts that $g(z)$ (which corresponds to $\sqrt{\Omega}\sigma(z)$ in Theorem \ref{thm:factorization}) satisfies
\[
p(e^{i\theta})=|g(e^{i\theta})|^2
\]
in agreement with \eqref{eq:factor}.
Moreover, since the roots of the polynomial $g(z)$ are in the complement of the unit disc, both $g(z)$ as well as its inverse are analytic in the closed unit disc, in agreement with Theorem \ref{thm:factorization}.
\black

In the matrix-valued setting of Theorem \ref{thm:factorization}, the ``zeros'' of $\sigma(z)$, that consist of the spectrum of $\sigma(z)^{-1}$, can be taken to be on the complement of the unit disc as stated.

\begin{remark} \rm
At first glance it appears that Theorem \ref{thm:factorization} follows from
 standard factorization theory for non-negative matrix-valued functions on the unit circle, based on the positive real lemma, see, e.g., \cite[Section 6]{LPbook}, or \cite{youla}. However, upon closer examination, standard spectral factorization results suggest that
\[
G(z)^*\Lambda G(z) = (D+zHG(z))^*(D+zHG(z))
\]
where $D+zHG(z)=D+zH(I-zA)^{-1}B$ is a canonical spectral factor of degree $n$. Thus, the main point of  Theorem~\ref{thm:factorization}, and the link to the Fej\'er-Riesz factorization, rests on the fact that there exists a matrix $L$ such that the canonical spetral factor is
\[
D+zH(I-zA)^{-1}B = L(I-zA)^{-1}B
\]
and hence, $D=L B$ and $H=LA$. In the rest of the paper we will assert and derive these identities in two different ways.
\end{remark}

\section{Motivation and notation}\label{sec:motivation}

Our interest in the type of generalized trigonometric polynomials that are introduced and studied below, stems from the moment problem that seeks to characterize admissible state covariances of linear dynamical systems. Specifically,
consider the linear, discrete-time dynamical system
\begin{align}\label{eq:dynamics}
x_{k+1}=Ax_k+Bu_k
\end{align}
where $x_k\in\mathbb C^n$ and $u_k\in\mathbb C^m$ are stationary (possibly, complex-valued) Gaussian processes, $A,B$ is a reachable pair of matrices and the spectrum of $A$ is contained in the open unit disc. Thus, the corresponding matrix-valued kernel in \eqref{eq:kernel}, $G(z)=(I-zA)^{-1}B$, is analytic in the closed unit disc.

If $\mu_{uu}(e^{i\theta})$ ($\theta\in[-\pi,\pi]$) denotes the matrix-valued spectral measure of the (stationary) input process $u_k$, then the state covariance is
\begin{align}\nonumber
R&=\mathbb E\{x_kx_k'\}\\\label{eq:int}
&=\int_{[0,2\pi]} G(e^{i\theta})d\mu_{uu}(\theta)G(e^{i\theta})^*.
\end{align}

The algebraic structure of stationary state covariances \cite{structure,ME} is dictated by the fact that
$R\in{\rm range}(\Gamma)$, with the map
\begin{eqnarray}
\Gamma &:& d\mu(\theta) \mapsto R=\int_{[0,2\pi]} G(e^{i\theta})d\mu(\theta)G(e^{i\theta})^*
\end{eqnarray}
acting on Hermitian-valued measures on $[0,2\pi]$; it turns out that
\begin{eqnarray}
{\rm range}(\Gamma)=\{R\in \mathbb C^{n\times n}\mid R-ARA^* = BX+X^* B^*, \mbox{ for some }X\in \mathbb C^{m\times n}\}.
\end{eqnarray}
In the converse direction, for any $R\in {\rm range}(\Gamma)\cap \mathbb S_+^{n\times n}$, with $\mathbb S_+^{n\times n}$ the cone of Hermitian $n\times n$ non-negative matrices, there exists a non-negative measure $d\mu_{uu}(\theta)$ so that \eqref{eq:int} holds \cite{structure,ME}.

In the special case when $R>0$, a spectral measure that satisfies \eqref{eq:int} can be selected to be in the form (see \cite{IT})
\begin{equation}\label{eq:denominator}
d\mu(\theta)=(G(e^{i\theta})^*\Lambda G(e^{i\theta}))^{-1}d\theta.
\end{equation}
This particular solution occupies a unique place amongst all solutions to \eqref{eq:int}, as it maximizes the entropy functional 
\[
\int_{[0,2\pi]}\log(\det (\dot\mu(\theta))) d\theta,
\]
with $\dot \mu$ denoting the almost everywhere defined derivative of $\mu$. 
The matrix $\Lambda$ that appears in the moment problem described above show up as a Lagrange multiplier, see e.g., \cite{GL1}.
The generalized trigonometric polynomial
\[
G(z)^*\Lambda G(z)
\]
for $z=e^{i\theta}$, which appears as the denominator in \eqref{eq:denominator} is the object of study at the present.
The adjoint map of $\Gamma$ is
\[
\Gamma^*:\Lambda \mapsto G(z)^*\Lambda G(z),
\]
and therefore, $\Lambda$ may always be taken to belong to ${\rm range}(\Gamma)=({\rm null}(\Gamma^*))^\perp$.

We note that the moment problem \eqref{eq:int} lies in the center of a rather rich and timely development in spectral analysis of vector-valued time series, e.g., see \cite{Gharmonic,ferrante1,ferrante2,jiang,zorzi,baggio,karlsson,zhu} and the references therein.

\section{Geometry and factorization of non-negative trigonometric polynomials}
 
We continue with a key lemma of independent interest on the geometric structure of the dual to a positive cone that is specified by algebraic constraints. 
%
The lemma implies (Corollary \ref{cor1}) the
existence of factorizations for matrix-valued non-negative generalized trigonometric polynomials, i.e., of the form $G(z)^*\Lambda G(z)$, for matrices $\Lambda$ that may not necessarily be non-negative definite.

\subsection{Duality and linear structure}

\newcommand{\dual}{{\rm dual}}
Let $\mathfrak C$ to be a self-dual non-negative cone, i.e., this is a subset of an inner product space $\mathfrak S$ that which is closed under addition and multiplication by non-negative scalars, and satisfies
\[
\mathfrak C=\mathfrak C^\dual.
\]
Throughout $\mathfrak C^\dual$ denotes the dual cone which is defined by
\[
\mathfrak C^\dual :=\{Q\in \mathfrak S \mid \langle Q,P\rangle \geq 0,\; \forall P \in \mathfrak C\}.
\]
Consider now a linear subspace $\mathcal S\subset \mathfrak S$, and let
$\Pi_{\mathcal S}$ denote
the orthogonal projection onto $\mathcal S$. It is of interest to consider a subset of $\mathfrak C$ whose elements are specified to satisfy linear constraints. Specifically, we consider 
$\mathcal S\cap \mathfrak C$ which is in itself a non-negative cone and its dual, in $\mathcal S$, is
\[
(\mathcal S\cap \mathfrak C)^\dual =\{Q\in \mathcal S\mid \langle Q,P\rangle\geq 0,\;\forall P\in \mathcal S\cap \mathfrak C\}.
\]
The following holds.
\begin{lemma}\label{ref:lemma1}
$(\mathcal S\cap \mathfrak C )^\dual=\Pi_{\mathcal S} \mathfrak C$.
\end{lemma}
\begin{proof}
First note that $\Pi_{\mathcal S} \mathfrak C \subseteq  (\mathcal S\cap \mathfrak C)^\dual$. To see this take a $\Lambda=\Pi_\mathcal S P\in \Pi_{\mathcal S} \mathfrak C$, with $P\in \mathfrak C$, and for an arbitrary $Q\in \mathfrak C\cap \mathcal S$, observe that since $\Pi_\mathcal S Q=Q$,
\[
\langle \Lambda, Q\rangle = \langle \Pi_\mathcal S P, Q\rangle = \langle P,Q\rangle \geq 0.
\]
Thus, $\Lambda \in (\mathcal S\cap \mathfrak C)^\dual$. This proves this first claim.

In order to establish the reverse inclusion, $(\mathcal S\cap \mathfrak C )^\dual\subseteq\Pi_{\mathcal S} \mathfrak C$, we prove instead the equivalent containment $\mathcal S\cap \mathfrak C \supseteq(\Pi_{\mathcal S} \mathfrak C)^\dual$. To this end, consider an arbitrary
element $Q\in (\Pi_{\mathcal S} \mathfrak C)^\dual$. It is readily seen that besides
$Q\in\mathcal S$,
\[
\langle Q,\Pi_\mathcal S P\rangle \geq 0 
\]
for all $P\in\mathfrak C$. But $\Pi_\mathcal S Q=Q$ and therefore, $\langle Q,P\rangle \geq 0$ for all $P\in \mathfrak C$. Since $ \mathfrak C$ is self dual, $Q\in  \mathfrak C$, which completes the proof.
\end{proof}

\subsection{Factorization of non-negative matrix polynomials.} Our usage of the lemma will rest on specializing to the case where
\[
\mathfrak C= \mathbb S_+,
\]
 the space of non-negative $n\times n$ Hermitian matrices (throughout $\mathbb S$ denotes Hermitian matrices), while the subspace
 \[
 \mathcal S =\{R\in\mathbb C^{n\times n} \mid R-ARA^* = BX+X^* B^*, \mbox{ for some }X\in \mathbb C^{m\times n}\},
 \]
which is precisely the range of $\Gamma$, d`d in Section \ref{sec:motivation}.
In this case, the adjoint operator to $\Gamma$ is
\[
\Gamma^*\,:\, \Lambda \in \mathbb S\to G(z)^*\Lambda G(z),
\]
and the orthogonal complement of its null space is of course ${\rm rangel}(\Gamma)=\mathcal S$. The following result readily follows.

\begin{prop} \label{prop:L+M}
Consider a generalized trigonometric polynomial $G(z)^*\Lambda G(z)$, for some $\Lambda \in \mathcal S:={\rm range}(\Gamma)$.
The following are equivalent:
\begin{itemize}
\item[i)] $G(e^{-i\theta})^*\Lambda G(e^{i\theta})\geq 0$ for all $\theta \in [-\pi,\pi]$.
\item[ii)] there exists an $M\in \mathcal S^\perp$ such that $\Lambda +M\geq 0$.
\end{itemize}
\end{prop}

\begin{proof}
If ii) holds, since $G(e^{-i\theta})^*M G(e^{i\theta})=0$ identically, then clearly
\[
G(e^{-i\theta})^*\Lambda G(e^{i\theta})=G(e^{-i\theta})^*(\Lambda+M) G(e^{i\theta})\geq 0.
\]
If i) holds, then by taking $\mathfrak C=\mathbb S_+$ and $\mathcal S$ as above and applying Lemma \ref{ref:lemma1}, ii) follows.
\end{proof}

\begin{corollary}\label{cor1}
Consider a generalized trigonometric polynomial $G(z)^*\Lambda G(z)$ such that 
\[
G(e^{-i\theta})^*\Lambda G(e^{i\theta})\geq 0, \mbox{ for all }\theta \in [-\pi,\pi].
\]
There is an $L=L^*\in\mathbb C^{n\times n}$ such that 
\begin{equation}
\label{p(z)factorization}
p(z)=G(z)^*L^*LG(z).
\end{equation}
\end{corollary}
While the above theory implies the existence of factorizations, we note however that it provides no immediate information on the zero-structure of $LG(z)$.

\section{1st Proof of Theorem~\ref{thm:factorization}}

We begin by observing that since $G(z)^*\Lambda G(z)> 0$ is rational and non-singular everywhere on the unit circle, then $(G(e^{i\theta})^*\Lambda G(e^{i\theta}))^{-1}> 0$ and integrable on the unit circle.
Thus, taking 
\begin{equation}\label{eq:form}
d\mu(\theta)=(G(e^{i\theta})^*\Lambda G(e^{i\theta}))^{-1}d\theta,
\end{equation}
we compute
\begin{equation}\label{eq:moment}
R=\int_{[0,2\pi]} G(e^{i\theta})d\mu(\theta)G(e^{i\theta})^*,
\end{equation}
which is Hermitian and positive definite.
We proceed to utilize some of the results in \cite{ME} on solutions to the matricial moment problem \eqref{eq:moment}.

The minimizer
\newcommand{\Ga}{{C_o}}
\[
C  = {\rm argmin}_X \{\trace(XRX^* ) \mid XB=I, X\in \mathbb C^{m\times n}\},
\]
is obtained explicitely in the form
\[
C  = (B^*  R^{-1} B)^{-1} B^* R^{-1}.
\]
Now set
\begin{align*}
\bar\Omega &= CRC^*\\
&= (B^*R^{-1}B)^{-1}, \mbox{ and} \\
\sigma(z)&=C G(z).
\end{align*}
Then, 
\[
\sigma(z)^{-1}=I-z CA(I-z(A-BCA))^{-1}B,
\]
and has no poles on the closed unit disc due to the fact that the state matrix
$(A-BCA)$ satisfies the Lyapunov equation
\[
R=B\bar\Omega B^*  +(A-BCA)R(A-BCA)^* .
\]

It is shown in \cite[Theorem 1]{ME} and  \cite[Theorem 1]{IT}, that the unique
maximizer of the strictly concave entropy functional
\[
\int_{-\pi}^{\pi} \log \det \left( \dot \mu (\theta) \right )d\theta,
\]
subject to the moment constraint \eqref{eq:moment}, is of the form
\eqref{eq:form}, i.e.,
\[
d\mu(\theta)=(G(e^{i\theta})^*Q G(e^{i\theta}))^{-1}d\theta,
\]
where the parameter $Q$ arises as the Lagrange multiplier for the constraint. Thus, it must be that $Q=\Lambda+M$, $M\in {\rm null}(\Gamma^*))^\perp$.
Moreover, it is shown that this unique maximizer is
\begin{align*}
d\mu(\theta)&= \left( \sigma(e^{i\theta})^{-1}\bar\Omega  (\sigma(e^{i\theta})^{-1})^*\right)d\theta\\
&=\left( (CG(e^{i\theta}))^* \bar\Omega^{-1} CG(e^{i\theta})\right)^{-1}d\theta.
\end{align*}
Hence, we recover the form claimed in Theorem \ref{thm:factorization}, namely, that
\begin{align*}
G(e^{i\theta}))^* \Lambda G(e^{i\theta}) &= G(e^{i\theta}))^* Q G(e^{i\theta})\\
&=\underbrace{G(z)^*C^*}_{\sigma(z)^*}\Omega \underbrace{CG(z)}_{\sigma(z)}
\end{align*}
for $Q=C^*  \Omega C$ and
$\Omega =\bar \Omega^{-1}$.
Conditions i) and iii) hold. Uniqueness in ii)  follows since, by
\cite[Theorem 1]{ME}, the maximizer of the entropy functional is of the required form.
 \hfill $\Box$

\black
\section{2nd Proof of Theorem~\ref{thm:factorization}}

By Proposition~\ref{prop:L+M} there is an $L\in\mathbb C^{n\times n}$ such that 
\begin{equation}
\label{p(z)}
p(z)=G(z)^*L^*LG(z)=W(z)W(z)^*,
\end{equation}
where $W(z):=G(z)^*L^*$ is given by
\begin{align}
   W(z) &=zB^*(zI-A^*)^{-1}L^* =B^*(zI-A^*+A^*)(zI-A^*)^{-1}L^*\notag \\
    &=B^*A^*(zI-A^*)^{-1}L^* +B^*L^* ,\label{Wprel}
\end{align}
which is an $m\times n$ stable spectral factor of $p(z)$. Since thus $p(z)$ is para-Hermitian, it is a spectral density, so there is additive decomposition 
\begin{equation}
\label{p=F+F*}
p(z)= F(z)+F(z)^*,
\end{equation}
where $F(z)$ is an $m\times m$ positive real function with the structure
\begin{equation}
\label{F}
F(z)=\tfrac12 \Lambda_0 + B^*A^*(zI-A^*)^{-1}\bar{B},
\end{equation}
where $(A^*,\bar{B})$ is a reachable pair \cite[p. 199]{LPbook}. In fact, by the Positive Real Lemma \cite[Theorem 6.7.4]{LPbook}, $F(z)$ is positive real if and only if there is a symmetric $n\times n$ matrix $P$, which will be positive definite, such that 
\begin{equation}
\label{M}
M(P):= \begin{bmatrix}P-A^*PA & \bar{B}-A^*PAB\\ \bar{B}^* -B^*A^*PA& \Lambda_0 -BA^*PAB\end{bmatrix}\geq 0 ,
\end{equation} 
in terms of which 
\begin{equation}
\label{spectralfactorization}
p(z)=\begin{bmatrix}B^*A^*(zI-A^*)^{-1}&I\end{bmatrix}M(P)\begin{bmatrix}(z^{-1}I-A)^{-1}AB\\I\end{bmatrix}.
\end{equation}
In view of \eqref{M} there is a rank factorization 
\begin{equation}
\label{MKD}
M(P)=\begin{bmatrix}K^*\\D^*\end{bmatrix}\begin{bmatrix}K&D\end{bmatrix}
\end{equation}
yielding via \eqref{spectralfactorization} the spectral factor 
\begin{equation}
\label{W}
W(z)=B^*A^*(zI-A^*)^{-1}K^* + D^*.
\end{equation}
There is one such spectral factor for each $P$ satisfying \eqref{M} \cite[Section 6.7.1]{LPbook}. 

Now, for the proof of Theorem~\ref{thm:factorization}, we need the $m\times m$ outer spectral factor
\begin{equation}
\label{Wminus}
W_-(z)=B^*A^*(zI-A^*)^{-1}K_-^* + D_-^*,
\end{equation}
 which corresponds to $P=P_-$ having the property that $P_-\leq P$ for all $P$ satisfying \eqref{M}. 
The positive definite matrix $P_-$ is the stabilizing solution of of the Algebraic Riccati Equation 
\begin{equation}
\label{ARE}
P=A^*PA+( \bar{B}^* -B^*A^*P)(\Lambda_0 -B^*A^*PAB)^{-1} (\bar{B}^* -B^*A^*P)^* ,
\end{equation}
which can be obtained as the limit
\begin{equation}
\label{RE}
P_-=\lim_{k\to\infty}P_k
\end{equation}
of the sequence obtained from the matrix Riccati equation
\begin{equation}
\label{ }
P_{k+1}=A^*P_kA+( \bar{B}^* -B^*A^*P_k)(\Lambda_0 -B^*A^*P_kAB)^{-1} (\bar{B}^* -B^*A^*P_k)^*,\quad P_0=0 
\end{equation}
  \cite[p. 470]{LPbook}. 
The stable spectral factor \eqref{Wminus} is outer and has  all its zeros in the open unit disc.

Let us first consider the special case when the matrix $A$ is nonsingular. Then $G(0)=B$ and $G^*(0)=0$, so $p(0)=0$.
Hence $W_-(0)=0$ as well, so it follows from \eqref{Wminus} that
\begin{equation}
\label{Dminus}
D_-=K_-B.
\end{equation}
Then reversing the  calculation leading to \eqref{Wprel} implies that 
\begin{equation}
\label{Wminus3}
W_-(z)=zB^*(zI-A^*)^{-1}K_-^*.
\end{equation}
Moreover,
\begin{displaymath}
M(P)=\begin{bmatrix}I\\B^*\end{bmatrix}K_-^*K_-\begin{bmatrix}I&B\end{bmatrix},
\end{displaymath} 
which yields in turn
\begin{subequations}
\begin{align}
   P_- &=A^*P_-A +K_-^*K_-   \label{Lyapunov}\\
   \bar{B}^* &=B^*A^*P_-A +B^*K_-^*K_- =B^*P_-\\
   \Lambda_0 &=B^*A^*P_-AB +B^*K_-^*K_-B=B^*PB =  \bar{B}^*B. \label{C0}
\end{align}
\end{subequations}
Then introduce  matrices $C\in \mathbb R^{m\times n}$ and $\Omega\in\mathbb R^{m\times m}$ such that 
 \begin{displaymath}
C^*\Omega C=K_-^*K_-\quad \text{and} \quad CB=I.
\end{displaymath}
Then $B^*K_-^*K_-B =B^*C^*\Omega CB = \Omega$, so, by \eqref{Lyapunov} and \eqref{C0}, 
\begin{equation}
\label{Omega}
\Omega=\Lambda_0 -B^*A^*P_-AB,
\end{equation}
which is the unique $\Omega$ in the theorem, and 
\begin{equation}
\label{COmega}
C^*\Omega C=P_-  -A^*P_-A.
\end{equation}
Clearly $C$ is uniquely defined. Consequently, since $W_-(z)=G(z)^*K_-^*$, it follows from \eqref{Wminus3} that
\begin{equation}
\label{CG}
CG(z)=C(I-zA)^{-1}B,
\end{equation}
which has all its zeros in the complement of the unit disc, since $W_-(z)$ is outer.

Next, consider the general case when $A$ might be singular. Then it folllows from \eqref{Wminus} and \eqref{Wprel} that
\begin{subequations}
\begin{equation}
\label{Wminus2}
W_-(z)=zB^*(zI-A^*)^{-1}K_-^* +\Psi ,
\end{equation}
where
\begin{equation}
\label{ }
\Psi=D_- -B^*K_-^* ,
\end{equation}
\end{subequations}
which, via \eqref{MKD}, is a continuous function of $P_-$. Now, for $\varepsilon >0$,  consider the positive real function 
\begin{equation}
\label{F}
F_\varepsilon(z)=\tfrac12 \Lambda_0 + B^*(A+\varepsilon I)^*(zI-(A+\varepsilon I)^*)^{-1}\bar{B},
\end{equation}
which, by the procedure above for nonsingular $A$,  corresponds to the spectral factor
\begin{equation}
\label{Wminus2}
W_\varepsilon(z)=zB^*(zI-A^*)^{-1}K_\varepsilon^* +\Psi_\varepsilon ,
\end{equation}
with $P_\varepsilon$ the stabilizing solution of the corresponding algebraic Riccati equation. From \cite{Delchamps} \black
we have\footnote{ The Riccati equation in \cite{Delchamps} is the one appearing in Kalman filtering, properly normalized. It is in terms of $P-P_-$, and is easily reformulated to our present setting \cite[Section 12.2]{LPbook}.} that $P_\varepsilon$ is continuous in $\varepsilon$, and hence so is also $\Psi_\varepsilon$. However, since $A+\varepsilon I$ is nonsingular, $\Psi_\varepsilon=0$ for $\varepsilon >0$, and hence the limit $\Psi_0$ is zero as well, yielding \eqref{Wminus3}. 

\black

\section*{Acknowledgment}
This research was partially supported by NSF under grants 1807664, 1839441, and
AFOSR under grant FA9550-20-1-0029.

\black


\begin{thebibliography}{99}
 \bibitem{baggio} G. Baggio and A. Ferrante, ``On the factorization of rational discrete-time spectral densities,'' IEEE Transactions on
Automatic Control, vol. 61, no. 4, pp. 969-981, 2016.

\bibitem{Delchamps} D.F. Delchamps, ``Analytic stabilization and the algebraic Riccati equation", {\em The 22nd IEEE Conference on Decision and Control}, 1983,  pp. 1396-1401.

 \bibitem{ferrante1} A. Ferrante, M. Pavon, and F. Ramponi, ``Hellinger versus Kullback-Leibler multivariable spectrum approximation,''  {\em IEEE Transactions on Automatic Control,} vol. 53, no. 4, pp. 954-967, 2008.

\bibitem{ferrante2} A. Ferrante, C. Masiero, and M. Pavon, ``Time and spectral domain relative entropy: A new approach to multivariate spectral estimation,'' {\em IEEE Transactions on Automatic Control,} vol. 57, no. 10, pp. 2561-2575, 2012.

\bibitem{structure} Georgiou, TT. "The structure of state covariances and its relation to the power spectrum of the input." IEEE Transactions on Automatic Control 47.7 (2002): 1056-1066.

\bibitem{ME} Georgiou, TT. "Spectral analysis based on the state covariance: the maximum entropy spectrum and linear fractional parametrization." IEEE transactions on Automatic Control 47.11 (2002): 1811-1823.

\bibitem{IT} Georgiou, TT. "Relative entropy and the multivariable multidimensional moment problem." IEEE Transactions on Information Theory 52.3 (2006): 1052-1066.

\bibitem{GL1}
T.~T. Georgiou and A.~Lindquist, ``Kullback-Leibler approximation of spectral
  density functions,'' {\em IEEE Transactions on Information Theory}, vol.~49,
  no.~11, pp. 2910--2917, 2003.
  
 \bibitem{GZ} Grenander U, Szeg\"o G., {\em Toeplitz forms and their applications}, Univ of California Press, 1958.

 \bibitem{Gharmonic} T.T. Georgiou, ``Spectral Estimation by Selective Harmonic Amplification,'' {\em IEEE Trans. on Automatic Control}, 46(1): 29-42, January 2001.
 
 \bibitem{jiang} X. Jiang, Z.-Q. Luo, and T. T. Georgiou, ``Geometric methods for spectral analysis,'' {\em  IEEE Transactions on Signal Processing}, vol. 60, no. 3, pp. 1064-1074, 2012.
 
 \bibitem{karlsson} J. Karlsson, A. Lindquist, and A. Ringh, ``The multidimensional moment problem with complexity constraint,'' {\em Integral Equations and Operator Theory}, 84 (2016), pp. 395-418.
 
 \bibitem{LPbook} Lindquist A and Picci G., {\em Linear stochastic systems: A Geometric Approach to Modeling, Estimation and Identification}, Berlin Heidelberg: Springer, 2015.

\bibitem{youla} Youla, D. ``On the factorization of rational matrices,'' {\em IRE Transactions on Information Theory} 7.3 (1961): 172-189.

\bibitem{zhu} B. Zhu and G. Baggio, ``On the existence of a solution to a spectral estimation problem {\em \`a la} Byrnes-Georgiou-Lindquist,'' {\em IEEE Transactions on Automatic Control}, 64 (2019), pp. 820-825.
 
\bibitem{zorzi} M. Zorzi, ``Multivariate spectral estimation based on the concept of optimal prediction,'' {\em  IEEE Transactions on Automatic Control}, vol. 60, no. 6, pp. 1647-1652, June 2015.



\end{thebibliography}
\end{document}